\documentclass[11pt]{article}

\usepackage{amsmath,graphicx}

\title{Homogeneous sets, clique-separators, critical graphs, and optimal $\chi$-binding functions}

\author{Christoph Brause, Maximilian Gei\ss er, and Ingo Schiermeyer\\
\small Institute for Discrete Mathematics and Algebra\\[-0.8ex]
\small TU Bergakademie Freiberg\\[-0.8ex] 
\small Germany}

\newcommand{\B}{banner}
\newcommand{\Bu}{bull}

\newcommand{\Di}{diamond}
\newcommand{\G}{gem}
\newcommand{\Ha}{co\textnormal{-}banner}

\newcommand{\Pa}{paw}
\newcommand{\Para}{paraglider}

\newcommand{\Q}[1]{Q[#1]}
\newcommand{\f}{f^\star_{\{3K_1\}}}
\newcommand{\fk}{f^\star_{\{2K_2\}}}
\newcommand{\dist}{{\rm dist}}

\newcommand{\ql}[1]{\vartriangleleft_\chi^{#1}}

\usepackage{enumerate}
\usepackage{amsthm,amssymb}

\usepackage[left=2.5cm,right=2.5cm,top=2.5cm,bottom=2.5cm]{geometry}

\usepackage{subcaption}
\usepackage{tikz}
\usetikzlibrary{calc}
\usetikzlibrary{decorations.pathreplacing}
\tikzstyle{snode}=[circle ,draw=black,fill=white,thick, inner sep=0pt ,minimum size=1.4mm]
\tikzstyle{bnode}=[circle ,draw=black,fill=black,thick, inner sep=0pt ,minimum size=1.4mm]

\theoremstyle{plain}
\newtheorem{theorem}{Theorem}
\newtheorem{lemma}[theorem]{Lemma}

\newtheorem{observation}[theorem]{Observation}

\theoremstyle{definition}

\theoremstyle{plain}

\newtheorem*{sthm}{The Strong Perfect Graph Theorem}
\newtheorem{cor}[theorem]{Corollary}
\newenvironment{non}[1]{\trivlist \item [\hskip \labelsep {\bf #1}]\ignorespaces\it}{\endtrivlist}

\begin{document}

\maketitle


\begin{abstract}
Given a set $\mathcal{H}$ of graphs, let $f_\mathcal{H}^\star\colon \mathbb{N}_{>0}\to \mathbb{N}_{>0}$ be the optimal $\chi$-binding function of the class of $\mathcal{H}$-free graphs, that is,
$$f_\mathcal{H}^\star(\omega)=\max\{\chi(G): G\text{ is } \mathcal{H}\text{-free, } \omega(G)=\omega\}.$$
In this paper, we combine the two decomposition methods by homogeneous sets and clique-separators in order to determine optimal $\chi$-binding functions for subclasses of 
$P_5$-free graphs and of $(C_5,C_7,\ldots)$-free graphs. In particular, we prove the following for each $\omega\geq 1$:
\begin{enumerate}
\item[(i)] $f_{\{P_5,\B\}}^\star(\omega)=\f(\omega)\in \Theta(\omega^2/\log(\omega)),$
\item[(ii)] $f_{\{P_5,\Ha\}}^\star(\omega)=f^\star_{\{2K_2\}}(\omega)\in\mathcal{O}(\omega^2),$
\item[(iii)] $f_{\{C_5,C_7,\ldots,\B\}}^\star(\omega)=f^\star_{\{C_5,3K_1\}}(\omega)\notin \mathcal{O}(\omega),$ and
\item[(iv)] 
$f_{\{P_5,C_4\}}^\star(\omega)=\lceil(5\omega-1)/4\rceil.$
\end{enumerate}

We also characterise, for each of our considered graph classes, all graphs $G$ with $\chi(G)>\chi(G-u)$ for each $u\in V(G)$. From these structural results, Reed's conjecture -- relating chromatic number, clique number, and maximum degree of a graph -- follows for $(P_5,\B)$-free graphs.
\end{abstract}

\section{Introduction}

The study of $\chi$-binding functions for (hereditary) graph classes is one of the central problems in chromatic graph theory. Motivated by the Strong Perfect Graph Conjecture of Berge~\cite{Be}, Gy\'arf\'as~\cite{G} introduced these upper bounds on the chromatic numbers of graphs that belong to a certain graph class. A lot of results have been published in the last decades in this particular field of graph theory (we refer the reader to the surveys of Schiermeyer and Randerath~\cite{RS}, and Scott and Seymour~\cite{SS}). Our contribution in this paper is a combination of decompositions by homogeneous sets and clique-separators which allows us to determine exemplary optimal $\chi$-binding functions for subclasses of $P_5$-free graphs as well as for subclasses of $(C_5,C_7,\ldots)$-free graphs.

The organisation of this paper is as follows: We continue in this section by motivating and presenting our main results as well as an introduction into notation and terminology. In Section~2, we prove some preliminary results that are used in later proofs while in the remaining sections our main results are proven.  We introduce our main technique in Section~3, deal with $\Ha$-free graphs in Section~4, with $\B$-free graphs in Section~5, and with $C_4$-free graphs in Section~6.

\subsection{Motivation and contribution}\label{sec:mot}

We use standard notation and terminology, and note that each of the considered graphs is simple, finite, and undirected in this paper. 
Some particular graphs are depicted in Fig.\,\ref{fig_forbidden_induced_Subgraphs}, and we denote a path and a cycle on $n$ vertices by $P_n$ and $C_n$, respectively.
Additionally, given graphs $G,H_1,H_2,\ldots $, the graph $G$ is \emph{$(H_1,H_2,\ldots)$-free} if $G-S$ is non-isomorphic to $H$ for each $S\subsetneq V(G)$ and each $H\in\{H_1,H_2,\ldots\}$.

\begin{figure}[h]
\centering
\begin{subfigure}[b]{0.085\textwidth}
\centering
\begin{tikzpicture}
\foreach \x in{1,2,3}
	\draw ({360/3*\x+90}:0.436) node[bnode]{};
\end{tikzpicture}
\subcaption{$3K_1$}
\end{subfigure}
\begin{subfigure}[b]{0.085\textwidth}
\centering
\begin{tikzpicture}
	\draw (0,0) node[bnode]{}--(0,-0.67) node[bnode]{};
	\draw (0.66,0) node[bnode]{}--(0.67,-0.67) node[bnode]{};
\end{tikzpicture}
\subcaption{$2K_2$}
\end{subfigure}
\begin{subfigure}[b]{0.07\textwidth}
\centering
\begin{tikzpicture}
	\draw (0,0) node[bnode]{}--(0,-0.67) node[bnode]{}--(0.66,-0.67) node[bnode]{}--(0.67,0) node[bnode]{}--(0,0);
\end{tikzpicture}
\subcaption{$C_4$}
\end{subfigure}
\begin{subfigure}[b]{0.09\textwidth}
\centering
\begin{tikzpicture}
	\draw (0.33,0.67) node[bnode]{}--(0.67,0) node[bnode]{}--(1,0.67) node[bnode]{}--(1.33,0) node[bnode]{};
	\draw (0.67,0) node[bnode]{}--(1.33,0) node[bnode]{};
\end{tikzpicture}
\subcaption{$\Pa$}
\end{subfigure}
\\
\begin{subfigure}[b]{0.11\textwidth}
\centering
\begin{tikzpicture}
	\draw (0,0) node[bnode]{}--(0.67,0) node[bnode]{}--(1.07,0.4) node[bnode]{}--(1.47,0) node[bnode]{}--(1.07,-0.4) node[bnode]{}--(0.67,0);
\end{tikzpicture}
\subcaption{$\B$}
\end{subfigure}
\begin{subfigure}[b]{0.14\textwidth}
\centering
\begin{tikzpicture}
	\draw (0,0) node[bnode]{}--(0.33,0.67) node[bnode]{}--(0.67,0) node[bnode]{}--(1,0.67) node[bnode]{}--(1.33,0) node[bnode]{};
	\draw (0.67,0) node[bnode]{}--(1.33,0) node[bnode]{};
\end{tikzpicture}
\subcaption{$\Ha$}
\end{subfigure}
\begin{subfigure}[b]{0.11\textwidth}
\centering
\begin{tikzpicture}
	\draw (-0.67,0) node[bnode]{}--(-0.33,0.67)--(0.33,0.67)--(0.67,0)node[bnode]{};
	\draw (0,0) node[bnode]{}--(-0.33,0.67) node[bnode]{};
	\draw (0,0) node[bnode]{}--(0.33,0.67) node[bnode]{};
\end{tikzpicture}
\subcaption{$\Bu$}
\end{subfigure}
\begin{subfigure}[b]{0.07\textwidth}
\centering
\begin{tikzpicture}
\foreach \x in {1,2,3,4,5}
 	\draw ({360/5*\x+90}:0.436) node[bnode]{}--({360/5*(\x+1)+90}:0.436) node[bnode]{};
\end{tikzpicture}
\subcaption{$C_5$}
\end{subfigure}
\begin{subfigure}[b]{0.105\textwidth}
\centering
\begin{tikzpicture}
	\draw (0,0) node[bnode]{}--(-0.67,0) node[bnode]{}--(-0.33,0.67)--(0.33,0.67)--(0.67,0);
	\draw (0,0) node[bnode]{}--(-0.33,0.67) node[bnode]{};
	\draw (0,0) node[bnode]{}--(0.33,0.67) node[bnode]{};
	\draw (0,0) node[bnode]{}--(0.67,0) node[bnode]{};
\end{tikzpicture}
\subcaption{$\G$}
\end{subfigure}
\begin{subfigure}[b]{0.09\textwidth}
\centering
\begin{tikzpicture}
	\draw (0,0) node[bnode]{}--(0.33,0.67) node[bnode]{}--(0.67,0) node[bnode]{}--(1,0.67) node[bnode]{}--(1.33,0) node[bnode]{};
\end{tikzpicture}
\subcaption{$P_5$}
\end{subfigure}
\begin{subfigure}[b]{0.155\textwidth}
\centering
\begin{tikzpicture}
	\foreach \x in {0,0.67}{
		\foreach \y in {-0.33,0.33,1}{
			\draw (\x,0)node[bnode]{}--(\y,-0.67)node[bnode]{};
		}	
	}
	\draw (0.33,-0.67)node[bnode]{}--(1,-0.67)node[bnode]{};
\end{tikzpicture}
\subcaption{$\Para$}
\end{subfigure}
\begin{subfigure}[b]{0.1\textwidth}
\centering
\begin{tikzpicture}
\foreach \x in {1,2,3,4,5}{
 	\draw (0,0)node[bnode]{}--({360/5*\x+90}:0.67) node[bnode]{}--({360/5*(\x+1)+90}:0.67) node[bnode]{};
}
\end{tikzpicture}
\subcaption{$W_5$}
\end{subfigure}\newline
\\
\begin{subfigure}[b]{0.11\textwidth}
\centering
\begin{tikzpicture}[rotate=90]
	\draw (0.3,-0.8) node[bnode]{}--(0.67,0) node[bnode]{}--(1.17,0.3) node[bnode]{}--(1.67,0) node[bnode]{}--(1.17,-0.3) node[bnode]{}--(0.67,0);
	\draw (0.67,0) node[bnode]{}--(1.17,0.8) node[bnode]{}--(1.67,0) node[bnode]{}--(1.17,-0.8) node[bnode]{}--(0.67,0);
	\draw (1.17,0.8) node[bnode]{}--(1.17,-0.8) node[bnode]{};
\end{tikzpicture}
\subcaption{$Q[P_4]$}
\end{subfigure}

\caption{Most frequently used graphs}\label{fig_forbidden_induced_Subgraphs}
\end{figure}

A function $L\colon V(G)\to\mathbb{N}_{>0}$ is a \emph{proper colouring} if $L(u)\neq L(v)$ for each pair of adjacent vertices $u,v\in V(G)$ and, for simplicity, we say that each 
$k\in \{L(u):u\in V(G)\}$ is a \emph{colour}.  The smallest number of colours for which there is a proper colouring of $G$ is the \emph{chromatic number} of $G$, denoted by $\chi(G)$. 
It is obvious that each \emph{clique}, which is a set of pairwise adjacent vertices, needs to be coloured by pairwise different colours in a proper colouring. Thus, the \emph{clique number}, which is the largest cardinality of a clique in $G$ and that is denoted by $\omega(G)$, is a lower bound on $\chi(G)$. Since the beginnings of chromatic graph theory, researchers have been interested in relating these two invariants. For example, Erd\H os~\cite{E} showed that the difference could be arbitrarily large by proving that, for every two integers $g,k\geq 3$, there is a graph $G$ with chromatic number at least $k$ and girth at least $g+1$. Note that $G$ is of \emph{girth} at least $g+1$ if and only if $G$ is $(C_3, C_4,\ldots, C_g)$-free.
In contrast, it attracted Berge~\cite{Be} to study \emph{perfect} graphs, which are graphs, say $G$, that satisfy $\chi(G-S)=\omega(G-S)$ for each $S\subsetneq V(G)$. His research resulted in two famous conjectures, the Weak and the Strong Perfect Graph Conjecture. The first one, proven by Lov\'asz~\cite{L}, states that the complementary graph of a perfect graph is perfect. 
In contrast to the Weak Perfect Graph Conjecture, the Strong Perfect Graph Conjecture was open for a long time but is nowadays confirmed and known as the Strong Perfect Graph Theorem.
\begin{sthm}[Chudnovsky et al.~\cite{CRST}]
A graph $G$ is perfect if and only if $G$ and $\bar{G}$ are $(C_5,C_7,\ldots)$-free. 
\end{sthm}

Before Berge's conjectures were proven, many researchers introduced problems surrounding the world of perfect graphs. 
For example, the concept of $\chi$-binding functions for graph classes relates the chromatic and clique numbers of a graph. It was introduced by Gy\'arf\'as~\cite{G}, and says that, given a class $\mathcal{G}$ of graphs, a function $f\colon \mathbb{N}_{>0}\to\mathbb{N}_{>0}$ is a \emph{$\chi$-binding function} for $\mathcal{G}$ if $\chi(G-S)\leq f(\omega(G-S))$ for each $G\in\mathcal{G}$ and each $S\subsetneq V(G)$.
If such a function exists, then $\mathcal{G}$ is \emph{$\chi$-bounded}.
For a $\chi$-bounded class $\mathcal{G}$, the function $f^\star\colon \mathbb{N}_{>0}\to\mathbb{N}_{>0}$ with
$$\omega\mapsto\max\{\chi(G-S):G\in\mathcal{G},S\subsetneq V(G), \omega(G-S)=\omega\}$$
is the \emph{optimal $\chi$-binding function} of $\mathcal{G}$.
For brevity, given some graphs $H_1,H_2,\ldots$, we let $f^\star_{\{H_1,H_2,\ldots\}}$ denote the optimal $\chi$-binding function of the class of $(H_1,H_2,\ldots)$-free graphs if the latter one is $\chi$-bounded.

Gy\'arf\'as~\cite{G} observed from Erd\H{o}s' \cite{E} result that a $\chi$-binding function does not exist for the class of $(H_1,H_2,\ldots,H_k)$-free graphs whenever each of the given graphs $H_1,H_2,\ldots,H_k$ with $k\in\mathbb{N}_{>0}$ contains a cycle. In other words, for the $\chi$-boundedness of the class of $(H_1,H_2,\ldots,H_k)$-free graphs it is necessary that at least one of $H_1,H_2,\ldots,H_k$ is a forest. Furthermore, Gy\'arf\'as~\cite{G} and, independently, Sumner~\cite{S} conjectured that there is such an upper bound on the chromatic numbers of $H$-free graphs whenever $H$ is a forest.
For example, the class of $P_t$-free graphs for $t\geq 5$ has a $\chi$-binding function~(cf.\,\cite{G}) although all known upper bounds on $f^\star_{\{P_5\}}(\omega)$ and $f^\star_{\{P_5,C_5\}}$ are non-polynomial in~$\omega$~(cf.\,\cite{SSS,CSiv}). To the best of our knowledge, it is also unknown whether there is a polynomial $\chi$-binding function for the class of $(C_5,C_7,\ldots)$-free graphs (which is a short notation for the class of graphs each of which is $C_{2k+5}$-free for each $k\in\mathbb{N}_0$) although a double exponential one exists (cf.\,\cite{SSodd}).

It is rather interesting that $P_4$-free graphs are perfect by the Strong Perfect Graph Theorem but already for small superclasses such as $P_5$-free graphs and $(C_5,C_7,\ldots)$-free graphs, the best known $\chi$-binding functions are non-polynomial. Although it is unknown whether $f_{\{P_5\}}^\star$ and $f^\star_{\{C_5,C_7,\ldots\}}$ are bounded by a polynomial or not, there is a big difference in the order of magnitude between $f^\star_{\{P_4\}}$ on one hand, and $f_{\{P_5\}}^\star$ and $f^\star_{\{C_5,C_7,\ldots\}}$ on the other hand. In particular, modifying a result of \cite{BRSV}, we obtain the following lemma which we prove in Section~2 and from which we deduce that the classes of $P_5$-free graphs and of $(C_5,C_7,\ldots)$-free graphs do not have a linear $\chi$-binding function:
\begin{lemma}\label{lemma_main_non-linear}
Let $\mathcal{H}$ be a set of graphs and $h$ be an integer such that $\bar{H}$ has girth at most $h$ for each graph $H\in \mathcal{H}$.
If the class of $\mathcal{H}$-free graphs is $\chi$-bounded, then $f_\mathcal{H}^\star$ cannot be bounded from above by a linear function.
%
\end{lemma}

Since the orders of magnitude of $f_{\{P_5\}}^\star$ and $f^\star_{\{C_5,C_7,\ldots\}}$ are unknown, it is of interest to study subclasses of $P_5$-free graphs and subclasses of $(C_5,C_7,\ldots)$-free graphs.
For example, it has been proven
\begin{itemize}
\item $f^\star_{\{P_5,\Pa\}}(\omega)=
\left.\begin{cases}
f^\star_{\{P_5,C_3\}}(\omega)&\text{if } \omega\leq 2,\\
\omega &\text{if } \omega> 2\\
\end{cases}\right\}
=
\left.
\begin{cases}
3&\text{if } \omega= 2,\\
\omega &\text{if } \omega\neq 2\\
\end{cases}\right\}$ (cf.\,\cite{O,R} or \cite{RS}),
\item $f^\star_{\{P_5,\Di\}}(\omega)\leq \omega+1$ (cf.\,\cite{R}), 
\item $f^\star_{\{P_5,C_4\}}(\omega),f^\star_{\{P_5,\G\}}(\omega)\leq \lceil 5\omega/4\rceil$ (cf.\,\cite{CKS,CKMM}), 
\item $f^\star_{\{P_5,\Para\}}(\omega) \leq \lceil 3\omega/2\rceil$ (cf.\,\cite{HK}), and
\item $f^\star_{\{C_5,C_7,\ldots,\Bu\}}(\omega),f^\star_{\{P_5,\Bu\}}(\omega) \leq {\omega+1\choose 2}$ (cf.\,\cite{CSiv}).
\end{itemize}
We note that Randerath~\cite{R} characterised all non-bipartite $(P_5,C_3)$-free graphs and refer the reader to the survey of Schiermeyer and Randerath~\cite{RS} for additional results.

Our main contribution in this paper is an approach which allows us to determine optimal $\chi$-binding functions. 

\begin{theorem}\label{thm_main}
If $\omega\in\mathbb{N}_{>0}$, then
\begin{enumerate}[\rm (i)]
\item  $f_{\{P_5,\B\}}^\star(\omega)=\f(\omega),$ 
\item  $f_{\{P_5,\Ha\}}^\star(\omega)=f^\star_{\{2K_2\}}(\omega).$
\item  $f_{\{C_5,C_7,\ldots,\B\}}^\star(\omega)=f^\star_{\{C_5,3K_1\}}(\omega),$ and
\item
$f_{\{P_5,C_4\}}^\star(\omega)=
\left\lceil\frac{5\omega-1}{4}\right\rceil.
$
\end{enumerate}
\end{theorem}

It is worth pointing out that there are just a few
graph classes for which optimal $\chi$-binding functions are known. As the examples described above suggest, mostly one can only determine a $\chi$-binding function, and it is often a tough and challenging problem to determine the optimal one or at least its order of magnitude.

Note that Lemma~\ref{lemma_main_non-linear} implies that $f_{\{C_5,3K_1\}}^\star,f_{\{3K_1\}}^\star,$ and $f_{\{2K_2\}}^\star$ cannot be bounded from above by a linear function. In particular, by results of Kim~\cite{Kim} and Wagon~\cite{W},
$$f_{\{3K_1\}}^\star(\omega)\in\Theta\left(\frac{w^2}{\log(w)}\right)\quad\text{and}\quad f_{\{2K_2\}}^\star(w)\leq {w+1\choose 2}\in\mathcal{O}(\omega^2),$$
respectively.
We note that, by using a result of Gaspers and Huang~\cite{GH} and a very nice inductive proof, one can subtract~$2$ from the upper bound on $\fk$ for $\omega\geq 3$.

On our way to optimal $\chi$-binding functions, we characterise in parallel \emph{critical} graphs, i.e, graphs $G$ with $\chi(G-u)<\chi(G)$ for each $u\in V(G)$. For this purpose, a \emph{strong expansion} of a graph $G'$ is a graph $G$ for which there are a partition of $V(G)$ into non-empty cliques $S_1,S_2,\ldots,S_{|V(G')|}$ and a bijective function $f\colon \{S_1,S_2,\ldots,S_{|V(G')|}\}\to V(G')$ such that each vertex of $S_i$ is adjacent to each vertex of $S_j$ if $f(S_i)$ is adjacent to $f(S_j)$ and each vertex of $S_i$ is non-adjacent to each vertex of $S_j$ if $f(S_i)$ is non-adjacent to $f(S_j)$ for each distinct~$i,j$.

\begin{theorem}\label{thm_critical_graphs}
Let $G$ be a critical graph.
\begin{enumerate}[\rm (i)]
\item If $G$ is $(P_5,\B)$-free, then $G$ is $3K_1$-free.
\item If $G$ is $(P_5,\Ha)$-free, then $G$ is $2K_2$-free.
\item If $G$ is $(C_5,C_7,\ldots,\B)$-free, then $G$ is $(C_5,3K_1)$-free.
\item If $G$ is $(P_5,C_4)$-free, then $G$ is complete or a strong expansion of a graph $G'$ with $G'\in\{C_5,W_5\}$.
\end{enumerate}
\end{theorem}

An interesting open conjecture by Reed~\cite{Re} is that $\chi(G)$ can be bounded from above by $\lceil(\Delta(G)+\omega(G)+1)/2\rceil$, where $\Delta(G)$ denotes the \emph{maximum degree} of $G$, i.e.\,the largest number of vertices that have a common adjacent vertex. For example, this conjecture is proven for 
\begin{itemize}
\item $(C_5,C_7,\ldots)$-free graphs~\cite{AKS},
\item $3K_1$-free graphs~\cite{K,KS},
\end{itemize}
and, to the best of our knowledge, it is open for $2K_2$-free graphs.
 By Theorem~\ref{thm_critical_graphs} and the above results, we obtain the following corollary:

\begin{cor}\label{cor_reed}
If $G$ is $(P_5,\B)$-free, then 
$$\chi(G)\leq \left\lceil\frac{\Delta(G)+\omega(G)+1}{2}\right\rceil.$$
\end{cor}

\subsection{Notation and terminology}

In this section, we briefly introduce notation and terminology we use in this paper. 

Recall that we consider finite, simple, and undirected graphs only. For notation and terminology not defined herein, we refer to~\cite{BM}. A \emph{graph} $G$ consists of vertex set $V(G)$ and edge set $E(G)$, where each edge $e\in E(G)$ is a size two subset of $V(G)$. For notational simplicity, we write $uv$ instead of $\{u,v\}$ to denote an edge of $G$. The \emph{complementary graph} of $G$, denoted by $\bar{G}$, has vertex set $V(G)$ and edge set $\{uv:u,v\in V(G),u\neq v, uv\notin E(G)\}$.
A \emph{copy} of $G$ is a graph that is isomorphic to $G$.
Additionally, given two vertices $u,v\in V(G)$ and a set $S\subseteq V(G)$, we let $N_G(u)$ denote the set of neighbours of $u$, $N_G[u]=N_G(u)\cup\{u\}$, $N_G(S)$ be the set of all vertices of $V(G)\setminus S$ that have a neighbour in $S$, $N_G[S]=N_G(S)\cup S$, and $\dist_{G}(u,v)$ be the distance of $u$ and $v$ in $G$.
We also let $N^i_G(S)=\{u:\min\{\dist_G(u,s):s\in S\}=i\}$ for $i\geq 2$.
Observe that $\Delta(G)=\max\{|N_G(u)|:u\in V(G)\}$ is the \emph{maximum degree} of $G$.
Furthermore, a graph $H$ with $V(H)=V(G)$ and $E(H)\subseteq E(G)$ is a \emph{spanning subgraph} of $G$.

As usual, $\mathbb{N}_0$ and $\mathbb{N}_{>0}$ are the sets of non-negative and positive integers, respectively. For some integer $k\in\mathbb{N}_{>0}$ and some set $S$, the set $\{1,2,\ldots,k\}$ is denoted by $[k]$ and the power set of $S$ is denoted by $2^S$. Additionally, a function $f\colon\mathbb{N}_{>0}\to\mathbb{N}_{>0}$ is \emph{superadditive} if $f(s_1)+f(s_2)\leq f(s_1+s_2)$ for each $s_1,s_2\in\mathbb{N}_{>0}$.

Let $G$ be a graph and $q\colon V(G)\to\mathbb{N}_0$ be a function with $q(v)\geq 1$ for some vertex $v\in V(G)$, which we also call \emph{vertex-weight function}. 
Given a set $S$ of vertices of $G$, $G[S]$ is the graph with vertex set $S$ and edge set $E(G)\cap \{s_1s_2:s_1,s_2\in S\}$.
We say that $G[S]$ is the graph \emph{induced} by $S$ and $S$ \emph{induces} $G[S]$ in $G$. 
A similar notation is that of $G[q]$, which denotes the graph $G[\{u:u\in V(G),q(u)\geq 1\}]$.
Given an additional graph $H$, we say that $H$ is an \emph{induced subgraph} of $G$ if there is some set $S_H\subseteq V(G)$ of vertices such that $G[S_H]=H$.
Assuming $H$ to be an induced subgraph of $G$, we further define
$$q(S)=\sum_{s\in S}q(s)\quad\text{and}\quad q(H)=q(V(H)).$$
For simplicity in notation and terminology, we say that $q$ instead of the restriction of $q$ to $V(H)$ is a vertex-weight function of $H$.  

Given two graphs $G_1,G_2$ and an integer $k\geq 1$, we denote by $G_1\cup G_2$ the \emph{union} of $G_1$ and $G_2$, that is, $G_1\cup G_2$ has vertex set $V(G_1)\cup V(G_2)$ and edge set $E(G_1)\cup E(G_2)$, and by $kG_1$ a graph consisting of $k$ pairwise vertex disjoint copies of $G_1$.

In this paper, we mainly work with forbidden induced subgraphs. Thus, given two graphs $G,H$ and a family $\mathcal{H}$ of graphs, we say that $G$ is \emph{$H$-free} if no induced subgraph of $G$ is isomorphic to $H$, and that $G$ is \emph{$\mathcal{H}$-free} if $G$ is $H'$-free for each $H'\in\mathcal{H}$. Recall that \emph{$(H_1,H_2,\ldots)$-free} means $\mathcal{H}$-free with $\mathcal{H}=\{H_1,H_2,\ldots\}$.

Let again $G$ be a graph and $q\colon V(G)\to\mathbb{N}_0$ be a vertex-weight function.
Recall that a \emph{clique} of $G$ is a set of vertices which are pairwise adjacent. The \emph{$q$-clique number} of $G$, denoted by $\omega_q(G)$, is the largest integer $k$ for which there is a clique $S$ of $G$ with $q(S)=k$. An independent set $S$ of $G$ is a set of vertices which is a clique in $\bar{G}$, that is, the vertices of $S$ are pairwise non-adjacent in $G$. 
The \emph{$q$-independence number} of $G$, denoted by $\alpha_q(G)$, equals $\omega_q(\bar{G})$. 
A \emph{$q$-colouring} $L\colon V(G)\to 2^{\mathbb{N}_{>0}}$ is a set-mapping for which $|L(u)|=q(u)$ for each $u\in V(G)$. 
We note that the integers of $L(u)$ are also called \emph{colours} of $u$ for $u\in V(G)$, and we say that $L$ \emph{colours} the vertices of $G$.
In view of a simple notation, we let
$$L(S)=\bigcup_{s\in S}L(s)\quad\text{and}\quad L(H)=L(V(H))$$
for each set $S\subseteq V(G)$ and each induced graph $H$ of $G$. 
The $q$-colouring $L$ is \emph{proper} if each two adjacent vertices of $G$ receive disjoint sets of integers.
Moreover, $G$ is \emph{$k$-colourable} (with respect to $q$) for some integer $k\in\mathbb{N}_{>0}$ if there is some proper $q$-colouring $L\colon V(G)\to 2^{[k]}$. 
The smallest integer $k$ for which $G$ is $k$-colourable (with respect to $q$) is the \emph{$q$-chromatic number} of $G$, denoted by $\chi_q(G)$.
For the vertex-weight function $q\colon V(G)\to[1]$, we use the classical terminology of \emph{clique number}, \emph{independence number}, and \emph{chromatic number} instead of $q$-clique number, $q$-independence number, and $q$-chromatic number, and denote these graph invariants by $\omega(G)$, $\alpha(G)$, and $\chi(G)$, respectively. 
Furthermore, recall that $G$ is \emph{perfect} if $\omega(G')=\chi(G')$ for each induced subgraph $G'$ of $G$.

Given a class $\cal G$ of graphs, we recall that a function $f\colon \mathbb{N}_{>0}\to\mathbb{N}_{>0}$ is a \emph{$\chi$-binding function} if $\chi(G')\leq f(\omega(G'))$ for each graph $G\in\mathcal{G}$ and each induced subgraph $G'$ of $G$. 
If such a function exists, then $\mathcal{G}$ is \emph{$\chi$-bounded}. 
Since we are interested in $\chi$-bounded graph classes defined by a set, say $\mathcal{H}$, of forbidden induced subgraphs, we let $f_\mathcal{H}^\star$ denote the \emph{optimal} $\chi$-binding function of the class of $\mathcal{H}$-free graphs, that is, $f_\mathcal{H}^\star\colon \mathbb{N}_{>0}\to\mathbb{N}_{>0}$ is defined by
$$\omega\mapsto \max\{\chi(G):G\text{ is } \mathcal{H}\text{-free, }\omega(G)=\omega\}.$$

Let again $G$ be a graph. For two disjoint sets $A$ and $B$ of vertices, we let $E_G[A,B]$ denote the set of all edges between $A$ and $B$ in $G$. 
If $|E_G[A,B]|=|A|\cdot |B|$, then we say that $A$ is \emph{complete} to $B$ in $G$. If $|E_G[A,B]|=0$, then we say that $A$ is \emph{anti-complete} to $B$ in $G$.
Furthermore, in view of simplicity, we write $a$ is complete/anti-complete to $B$ instead of $\{a\}$ is complete/anti-complete to $B$.
A non-empty set $M$ of vertices of $G$ is a \emph{module} if $M$ is complete to $N_G(M)$ in $G$. We note that a module $M$ is a \emph{homogeneous set} if $1<|M|<|V(G)|$. The graph $G$ is \emph{prime} if there is no homogeneous set in $G$.
Let $k\geq 1$ be an integer, $G_1,G_2$ be two not necessarily connected induced subgraphs of $G$ with $G=G_1\cup G_2$ and $V(G_1)\setminus V(G_2),V(G_2)\setminus V(G_1)\neq\emptyset$, and $X_1$, $X_2$, \ldots, $X_k$ be $k$ pairwise vertex disjoint modules in $G$. If
\begin{itemize}
\item $X_i$ is complete to $X_j$ in $G$ for each distinct $i,j\in[k]$ and 
\item $V(G_1)\cap V(G_2)=X_1\cup X_2\cup \ldots \cup X_k$,
\end{itemize}
then $X_1\cup X_2\cup\ldots \cup X_k$ is a \emph{clique-separator of modules} in $G$.

Let $q,q'\colon V(G)\to\mathbb{N}_0$ be two vertex-weight functions of a graph $G$. We write $q'\vartriangleleft_\chi^G q$ if
$\chi_{q'}(G)=\chi_{q}(G)$, $q'(G)<q(G)$, and $q'(u)\leq q(u)$ for each $u\in V(G)$.
Additionally, $q$ is \emph{$\vartriangleleft_\chi^G$-minimal} if there is no vertex-weight function $q'\colon V(G)\to\mathbb{N}_0$ with $q'\vartriangleleft_\chi^G q$.
If $q$ is not $\vartriangleleft_\chi^G$-minimal, then there is a $\vartriangleleft_\chi^G$-minimal vertex-weight function $q''\colon V(G)\to\mathbb{N}_0$ 
with $q''\vartriangleleft_\chi^G q$ as $\chi_q(G)>0$ and $q(G)$ is finite.
We further note that a graph $G$ is \emph{critical} if $q\colon V(G)\to[1]$ is $\vartriangleleft_\chi^G$-minimal.

Let $G$ be a graph. 
An \emph{expansion} of $G$ is a graph $G'$
whose vertex set can be split into pairwise disjoint sets $\{X_u\}_{u\in V(G)}$ such that
\begin{itemize}
\item $X_u$ is empty or a clique,
\item $X_u$ is complete to $X_v$ if $uv\in E(G)$ and $X_u,X_v\neq \emptyset$, and
\item $X_u$ is anticomplete to $X_v$ if $uv\notin E(G)$  and $X_u,X_v\neq \emptyset$
\end{itemize}
for each two distinct vertices $u,v\in V(G)$.
If $q\colon V(G)\to\mathbb{N}_0$ is a vertex-weight function, then a \emph{$q$-expansion} of $G$ is an expansion of $G$ with corresponding vertex sets $\{X_u\}_{u\in V(G)}$ such that $|X_u|=q(u)$ for each $u\in V(G)$.
We note that, for a \emph{strong} expansion $G'$ of $G$, there is a vertex-weight function $q\colon V(G)\to\mathbb{N}_{>0}$ such that $G'$ is a $q$-expansion of $G$.

As usual, $C_n$, $K_n$, and $P_n$ denote a cycle, a complete graph, and a path of order $n$, respectively, and $K_{n,m}$ denotes a complete bipartite graph whose partite sets have sizes $n$ and $m$. Additionally, if $P\colon u_1u_2u_3u_4$ is a path on $4$ vertices and $F$ is an arbitrary graph that is vertex disjoint from $P$, then $\Q{F}$ is the graph obtained from $P\cup F$ by removing $u_3$ and adding all edges from $\{u_2,u_4\}$ to $V(F)$ (see Fig.~\ref{fig_forbidden_induced_Subgraphs} for an example where $F$ is a path on $4$ vertices as well).

Finally, recall that $G$ is \emph{$(C_5,C_7,\ldots)$-free} if $G$ is $C_{2k+5}$-free for each $k\in\mathbb{N}_0$.

\section{Preliminaries}

Before we proceed with preliminary results that are used in later proofs, we establish Lemma~\ref{lemma_main_non-linear} first. Recall its statement.
\begin{non}{Lemma~\ref{lemma_main_non-linear}}
Let $\mathcal{H}$ be a set of graphs and $h$ be an integer such that $\bar{H}$ has girth at most $h$ for each graph $H\in \mathcal{H}$.
If the class of $\mathcal{H}$-free graphs is $\chi$-bounded, then $f_\mathcal{H}^\star$ cannot be bounded from above by a linear function.
\end{non}
\begin{proof}
We may assume that the class of $\mathcal{H}$-free graphs is $\chi$-bounded. Furthermore, by definition, we have $h\geq 3$.
By a result of Bollob\'as \cite{B}, for each two integers $g,\Delta\geq 3$, there is a graph $G_{g,\Delta}$ of girth at least $g+1$ and maximum degree $\Delta$ for which 
$$\frac{\alpha(G_{g,\Delta})}{|V(G_{g,\Delta})|}< \frac{2\log(\Delta)}{\Delta}.$$
Hence, there is a series $\{\bar{G}_{h,i}\}_{i=3}^\infty$ such that, for each $i\geq 3$, $\bar{G}_{h,i}$ is a graph whose complementary graph is $G_{h,i}$.
For $i\geq 3$, 
note that $\bar{G}_{h,i}$ is $\mathcal{H}$-free. Since $G_{h,i}$ is $C_3$-free, it follows
$\alpha(\bar{G}_{h,i})=\omega(G_{h,i})\leq 2.$
Furthermore, $$\omega(\bar{G}_{h,i})<\frac{2\log(i)}{i}\cdot |V(\bar{G}_{h,i})|,$$
and so 
$$
\frac{i}{4\cdot \log(i)}\cdot \omega(\bar{G}_{h,i})<\frac{|V(\bar{G}_{h,i})|}{2}\leq \frac{|V(\bar{G}_{h,i})|}{\alpha(\bar{G}_{h,i})}\leq\chi(\bar{G}_{h,i}).
$$
Note that 
$$
\lim_{i\to\infty}\frac{i}{4\cdot \log(i)}=+\infty.
$$
Thus, $f_\mathcal{H}^\star$ cannot be bounded from above by a linear function.
\end{proof}

One central result in Lov\'asz'~\cite{L} proof of the Weak Perfect Graph Theorem is a lemma on so-called `perfect'-expansions. For our purpose the following weaker version suffices:
\begin{lemma}[part of a stronger result proven in Lov\'asz~\cite{L}]\label{preliminary_lemma_lovasz}
If $G$ is a perfect graph, then each expansion of $G$ is perfect.
\end{lemma}

We continue by an observation concerning the chromatic and clique numbers of $q$-expansions of a graph.

\begin{observation}\label{preliminary_obs_q_expansion}
If $G$ is a graph, $q\colon V(G)\to\mathbb{N}_0$ is a vertex-weight function, and $G'$ is a $q$-expansion of $G$, then 
$$\chi(G')=\chi_q(G)\quad\text{and}\quad\omega(G')=\omega_q(G).$$
\end{observation} 

Note that Observation~\ref{preliminary_obs_q_expansion} together with Lemma~\ref{preliminary_lemma_lovasz} implies $\chi_q(G)=\omega_q(G)$ for each perfect graph $G$ and each vertex-weight function $q\colon V(G)\to\mathbb{N}_0$. 

Following Narayanan and Shende~\cite{NS}, who proved
$$\chi(G)=\max\left\{\omega(G),\left\lceil\frac{|V(G)|}{\alpha(G)}\right\rceil\right\}$$
for each expansion $G$ of a cycle of length at least $4$, we can determine the $q$-chromatic number of a cycle of length $5$ by Observation~\ref{preliminary_obs_q_expansion}.

\begin{cor}\label{preliminary_cor_C5}
Let $\omega\in\mathbb{N}_{>0}$. If $C$ is a cycle of length $5$ and $q\colon V(C)\to\mathbb{N}_0$ is a vertex-weight function such that $\omega_q(C)=\omega$, then 
$$\chi_q(C)=\max\left\{\omega,\left\lceil\frac{q(C)}{2}\right\rceil\right\}\leq\left\lceil\frac{5\omega-1}{4}\right\rceil,$$
and this bound is tight.
\end{cor}
\begin{proof}
In view of Observation~\ref{preliminary_obs_q_expansion} and the result of Narayanan and Shende~\cite{NS}, it remains to show 
$$\left\lceil\frac{q(C)}{2}\right\rceil\leq\left\lceil\frac{5\omega-1}{4}\right\rceil$$
and that this bound is tight.
Renaming vertices if necessary, let us assume that\linebreak $C\colon c_1c_2c_3c_4c_5c_1$ is defined such that $\omega=q(\{c_1,c_2\})$ and $q(c_1)\geq q(c_2)$. Thus, $q(\{c_3,c_4\})\leq \omega$ and 
$q(c_5)\leq\lfloor\omega/2\rfloor$. Furthermore, for $n,m\in\mathbb{N}_0$ with $\omega=4n+m$ and $m<4$, we have 
\begin{align*}
\left\lceil\frac{q(C)}{2}\right\rceil&\leq
\omega+\left\lceil \frac{\left\lfloor\frac{\omega}{2}\right\rfloor}{2} \right\rceil
=\omega+\left.\begin{cases}
n&\text{if }m\leq 1\\ 
n+1&\text{if }m\geq 2\\ 
\end{cases}\right\}
= \omega+\left\lceil\frac{\omega-1}{4}\right\rceil
=\left\lceil\frac{5\omega-1}{4}\right\rceil.
\end{align*}
From this chain of inequalities it follows that the bound is tight if $$q(c_1)=q(c_3)=\lceil\omega/2\rceil\quad\text{and}\quad q(c_2)=q(c_4)=q(c_5)=\lfloor\omega/2\rfloor,$$ which completes our proof.
\end{proof}

\section{The technique}

In this section, we introduce the technique that can be used to prove Theorem~\ref{thm_main} and Theorem~\ref{thm_critical_graphs}. In view of its application to other graph classes, e.g.\;subclasses of $\Q{P_4}$-free graphs as in~\cite{BG}, we introduce all results in a general setting.

We concentrate next on our combination of homogeneous sets and clique-separators, namely the so-called clique-separators of modules. 
Note that each clique-separator is a clique-separator of modules. Having this observation in mind, the following lemma generalises the fact that critical graphs do not contain clique-separators since it implies that $G[q]$, for some $\ql{G}$-minimal vertex-weight function $q\colon V(G)\to \mathbb{N}_0$, does not contain a clique-separator of modules.

\begin{lemma}\label{preliminary_lemma_cutsets}
If $G,G_1,G_2$ are three graphs such that $G=G_1\cup G_2$ and $V(G_1)\cap V(G_2)$ is a clique-separator of modules in $G$, and $q\colon V(G)\to\mathbb{N}_0$ is a vertex-weight function,
then 
$$\chi_q(G)= \max\{\chi_q(G_1),\chi_q(G_2)\}\quad\text{and}\quad\omega_q(G)=\max\{\omega_q(G_1),\omega_q(G_2)\}.$$
\end{lemma}
\begin{proof}
Clearly, $\omega_q(G)\geq \max\{\omega_q(G_1),\omega_q(G_2)\}$. As there is no edge from $V(G_1)\setminus V(G_2)$ to $V(G_2)\setminus V(G_1)$ in $G$, equality follows.

Let $k\in\mathbb{N}_{>0}$ and $X,X_1,X_2\ldots,X_k\subseteq V(G)$ be sets such that $X=X_1\cup X_2\cup\ldots\cup X_k=V(G_1)\cap V(G_2)$ and $X_1,X_2,\ldots,X_k$ are the modules of $X$.

As $G_1$ and $G_2$ are subgraphs of $G$, we have $\chi_q(G)\geq \max\{\chi_q(G_1),\chi_q(G_2)\}$.
We colour $G[X]$ optimally, that is, there is a $q$-colouring $L_X\colon X\to 2^{[\chi_q(G[X])]}$ of $G[X]$.
As each $X_i$ is complete to $X\setminus X_i$ in $G[X]$, we find that $L_X$ uses $\chi_q(G[X_i])$ colours on $X_i$.
Furthermore, $L_X$ uses disjoint sets of colours for $X_i$ and for $X_j$ if $i,j\in[k]$ are distinct.
Let $j\in [2]$, and $L_j\colon V(G_j)\to 2^{[\chi_q(G_j)]}$ be a $q$-colouring of $G_j$. As each $X_i$ is a module in $G_j$, it follows that we can recolour the vertices of $X_i$ in $L_j$ by any $q$-colouring of $G[X_i]$ using colours from $L_j(X_i)$. By permuting colours and applying the above arguments for each $X_i$, we may assume that $L_j$ and $L_X$ coincide on $X$.
As there is no edge from $V(G_1)\setminus V(G_2)$ to $V(G_2)\setminus V(G_1)$ in $G$, it follows
$$\chi_q(G)= \max\{\chi_q(G_1),\chi_q(G_2)\},$$
which completes our proof.
\end{proof}

Note that $\Q{P_4}$ contains induced copies of $\B$ and $C_4$. In order to prove our main results, we wish to establish some preliminary results for $\Q{P_4}$-free graphs  but begin by considering modules of $\Q{F}$-free graphs, where $F$ is arbitrary and not necessarily related to $P_4$.

\begin{lemma}\label{preliminary_lemma_hom-sets}
If $F$ is a graph and $G$ is a connected $\Q{F}$-free graph, then, for each module $M$ in $G$, $G[M]$ is $F$-free or $N_G(M)$ is a clique-separator of modules, or $N_G^2(M)=\emptyset$.
\end{lemma}
\begin{proof}
Let us assume that $M$ is a module in $G$ such that $G[M]$ contains an induced copy of $F$ on vertex set $S$, and $N_G^2(M)\neq \emptyset$. 
As $N_G^2(M)\neq \emptyset$, we have $|M|<|V(G)|$.
We continue by showing that $N_G(M)$ is a clique-separator of modules.
Let $X_1,X_2,\ldots,X_k$ be the sets of vertices which induce the components of $\bar{G}[N_G(M)]$. 
Since 
$M\cup X_i$ is complete to $X_j$ in $G$ for each distinct $i,j\in[k]$ and $N_G^2(M)\neq \emptyset$, we may suppose, for the sake of a contradiction, that there is some $\ell\in[k]$ and a vertex $w\in N^2_G(M)$ for which $X_\ell\cap N_G(w)\neq \emptyset$ and $X_\ell\setminus N_G(w)\neq \emptyset$. Hence, by the connectivity of $\bar{G}[X_\ell]$, we may assume that $x_1\in X_\ell\cap N_G(w)$ and $x_2\in X_\ell\setminus N_G(w)$ are non-adjacent. Thus, $S\cup\{w,x_1,x_2\}$ induces a copy of $\Q{F}$, which contradicts our assumption that $G$ is $\Q{F}$-free. Thus, $X_\ell$ is a module in $G$, and $N_G(M)$ is a clique-separator of modules, which completes our proof.
\end{proof}

Let us focus on $\Q{P_4}$-free graphs next. It is rather interesting that every vertex-weight function of a $\Q{P_4}$-free graph can be nicely decomposed.

\begin{lemma}\label{preliminary_lemma_Q_P4}
If $G$ is a $\Q{P_4}$-free graph and $q\colon V(G)\to\mathbb{N}_0$ is a vertex-weight function, then there exist an integer $k\in\mathbb{N}_{>0}$, $k$ pairwise disjoint non-empty sets $M_1,M_2,\ldots, M_k\subseteq V(G[q])$, and $k$ $\ql{G}$-minimal vertex-weight functions $q_1,q_2,\ldots ,q_k\colon V(G)\to\mathbb{N}_0$ such that, for each $i,j\in [k]$,
\begin{enumerate}
\item[\rm (i)] $G[q_i]$ is a prime graph without clique-separators of modules and $V(G[q_i])\subseteq M_i$,
\item[\rm (ii)] $G[M_i]$ is a strong expansion of $G[q_i]$,
\item[\rm (iii)] $M_i$ is complete to $M_j$ in $G$ if $i$ and $j$ are distinct,
\item[\rm (iv)] $\omega_q(G[M_i])\geq \omega_{q_i}(G)$ with equality if $q$ is $\ql{G}$-minimal,
\item[\rm (v)] $\chi_q(G[M_i])=\chi_{q_i}(G)$ and
$$\chi_q(G)=\sum_{i=1}^k\chi_q(G[M_i]).$$
\end{enumerate}
\end{lemma}
\begin{proof}
In this proof, we consider pairs $(G,q)$ where $G$ is a $\Q{P_4}$-free graph and $q\colon V(G)\to\mathbb{N}_0$ is a vertex-weight function. Thus, by the definition of a vertex-weight function, we implicitly have $q(G)>0$ for each such pair. 
Furthermore, if $(G,q)$ is a pair that satisfies the statement of the lemma, then we say that $(G,q)$ is \emph{decomposable}.

We prove our lemma by contradiction.
For the sake of a contradiction, let us suppose that $(G,q)$ is a minimal counterexample to our lemma, that is, $(G,q)$ is not decomposable but each pair $(G',q')$ with either
$G'$ is an induced subgraph of $G$ with $G'\neq G$, or $G'=G$ and $|V(G[q'])|<|V(G[q])|$,  or $G'=G$ and $|V(G[q'])|=|V(G[q])|$ and $q'\ql{G} q$ is decomposable. 
We are now in a position to prove two properties of our minimal counterexample $(G,q)$.

We first prove $q(u)>0$ for each $u\in V(G)$.
If there is a vertex $u\in V(G)$ with $q(u)=0$, then, since $(G,q)$ is a minimal counterexample and $G[q]$ is an induced subgraph of $G$ with $G[q]\neq G$, we have that $(G[q],q)$ is decomposable, which also implies that $(G,q)$ is decomposable. The latter contradiction to our supposition on $(G,q)$ implies $G=G[q]$.

We next show that $q$ is $\ql{G}$-minimal by supposing, for the sake of a contradiction, the contrary. If $q'\colon V(G)\to\mathbb{N}_0$ is $\ql{G}$-minimal with $q'\ql{G}q$, then $(G,q')$ is decomposable into pairwise disjoint non-empty sets $M_1',M_2',\ldots,M_k'\subseteq V(G)$ and $\ql{G}$-minimal vertex-weight functions $q'_1,q'_2,\ldots ,q'_k\colon V(G)\to\mathbb{N}_0$ since $(G,q)$ is a minimal counterexample.
In what follows, we show that indeed $(G,q)$ is decomposable into the sets $M_1',M_2',\ldots,M_k'\subseteq V(G)$ and the vertex-weight functions $q'_1,q'_2,\ldots ,q'_k\colon V(G)\to\mathbb{N}_0$.
As $V(G[q'])\subseteq V(G[q])=V(G)$, it remains to prove (iv) and (v).
Since $q'\ql{G}q$ and $(G,q')$ is decomposable,
$$\chi_q(G[M_i'])\geq \chi_{q'}(G[M_i'])=\chi_{q_i'}(G)\quad\text{and}\quad \omega_q(G[M_i'])\geq \omega_{q'}(G[M_i'])\geq \omega_{q_i'}(G)$$
 for each $i\in[k]$.
In particular, (iv) follows.
It remains to prove (v). By definition, $\chi_{q'}(G)=\chi_q(G)$ as $q'\ql{G}q$.
As $M_i'$ is complete to $M_j'$ in $G$ for each distinct $i,j\in[k]$, we have
\begin{align*}\chi_q(G[M_i'])+\sum_{j\in[k]\setminus \{i\}} \chi_q(G[M_j'])&\leq \chi_q(G)=\chi_{q'}(G)=\sum_{i=1}^k \chi_{q'}(G[M_i'])\\
&\leq \chi_{q'}(G[M_i'])+\sum_{j\in[k]\setminus \{i\}} \chi_{q}(G[M_j']),
\end{align*}
and so $\chi_q(G[M_i'])\leq \chi_{q'}(G[M_i'])$ for each $i\in[k]$. 
Thus,
$$\chi_q(G[M_i'])=\chi_{q'}(G[M_i'])=\chi_{q_i'}(G)\quad\text{and}\quad \chi_{q}(G)=\chi_{q'}(G)=\sum_{i=1}^k\chi_{q'}(G[M_i'])=\sum_{i=1}^k\chi_q(G[M_i']).$$
As (v) follows, $(G,q)$ is decomposable, which contradicts our supposition on $(G,q)$. Therefore, we have that $q$ is $\ql{G}$-minimal.

As we have our two properties, we take an inclusion-wise minimal module $M_1$ in $G$ for which $N_G^2(M_1)=\emptyset$. Note that possibly $M_1=V(G)$. In what follows, we distinguish the two cases $M_1\neq V(G)$ and $M_1=V(G)$. We note that $G$ is connected as $q$ is $\ql{G}$-minimal.

\medskip

\noindent \textit{Case 1: $M_1\neq V(G)$}

\medskip

As $G$ is connected, $M_1$ is complete to $V(G)\setminus M_1$ in $G$.
For $S\in\{M_1,V(G)\setminus M_1\}$, let $q^S\colon V(G)\to\mathbb{N}_{0}$ be defined by
$$u\mapsto\begin{cases}
q(u)&\text{if } u\in S,\\
0&\text{if } u\notin S.\\
\end{cases}$$
Note that $q^{M_1}(G)>0$ and $q^{V(G)\setminus M_1}(G)>0$ as $M_1\neq\emptyset$, $M_1\neq V(G)$, and $G=G[q]$.
Furthermore, $\chi_q(G[M_1])=\chi_{q^{M_1}}(G)$ and $\chi_q(G-M_1)=\chi_{q^{V(G)\setminus M_1}}(G)$, and so
$$\chi_q(G)=\chi_{q^{M_1}}(G)+\chi_{q^{V(G)\setminus M_1}}(G).$$ Thus, $q^{M_1}$ and $q^{V(G)\setminus M_1}$ are $\ql{G}$-minimal since $q$ is $\ql{G}$-minimal. Hence, since \linebreak $|V(G[q^{M_1}])|,|V(G[q^{V(G)\setminus M_1}])|< |V(G[q])|$ and since $(G,q)$ is a minimal counterexample, we have that $(G,q^{M_1})$ and $(G,q^{V(G)\setminus M_1})$ are decomposable into pairwise disjoint non-empty sets \linebreak$M_1',M_2',\ldots,M_{k_1}'\subseteq V(G[q^{M_1}])$ and $M_{k_1+1}',M_{k_1+2}',\ldots,M_{k_1+k_2}'\subseteq V(q^{V(G)\setminus M_1})$ as well as $\ql{G}$-minimal vertex-weight functions $q_1',q_2',\ldots, q_{k_1}'\colon V(G)\to\mathbb{N}_0$ and $q_{k_1+1}',q_{k_1+2}',\ldots, q_{k_1+k_2}'\colon V(G)\to\mathbb{N}_0$, respectively.
In what follows, we show that indeed $(G,q)$ is decomposable into the sets \linebreak $M_1',M_2',\ldots,M_{k_1+k_2}'\subseteq V(G[q])$ and the vertex-weight functions  $q'_1,q'_2,\ldots ,q'_{k_1+k_2}\colon V(G)\to\mathbb{N}_0$.
As $V(G[q^{M_1}]),V(G[q^{V(G)\setminus M_1}])\subseteq V(G)$, it remains to prove (iii), (iv), and (v).
For (iii), it suffices to show that $M_i'$ is complete to $M_j'$ in $G$ for $i\in[k_1]$ and $j\in[k_1+k_2]\setminus[k_1]$ by the decompositions of $(G,q^{M_1})$ and $(G,q^{V(G)\setminus M_1})$. However, this fact follows from the observation that $M_1$ is complete to $V(G)\setminus M_1$ in $G$. Thus, (iii) is proven.
Additionally, since $q^S$ is $\ql{G}$-minimal, we have 
$$\omega_q(G[M_i'])=\omega_{q^S}(G[M_i'])=\omega_{q_i'}(G)\quad\text{and}\quad \chi_q(G[M_i'])=\chi_{q^S}(G[M_i'])=\chi_{q_i'}(G)$$
for each $i\in[k_1+k_2]$, and $S=M_1$ if $i\in[k_1]$ and $S=V(G)\setminus M_1$ if $i\notin[k_1]$. Thus, (iv) follows.
Furthermore,
$$\chi_q(G)=\chi_{q^{M_1}}(G)+\chi_{q^{V(G)\setminus M_1}}(G)=
\sum_{i=1}^{k_1}\chi_{q^{M_1}}(G[M_i'])+\sum_{i=k_1+1}^{k_1+k_2}\chi_{q^{V(G)\setminus M_1}}(G[M_i'])
=\sum_{i=1}^{k_1+k_2}\chi_q(G[M_i']).$$
Hence, (v) follows and $(G,q)$ is decomposable, which is a contradiction to our supposition on $(G,q)$.

\medskip

\noindent \textit{Case 2: $M_1=V(G)$}

\medskip
  
Recall that $G$ is connected and $q$ is $\ql{G}$-minimal, and so $G$ has no clique-separator of modules by Lemma~\ref{preliminary_lemma_cutsets}.

We show as a further property in this case that there is some integer $\ell\in\mathbb{N}_0$ and $\ell$ pairwise disjoint homogeneous sets $M_1',M_2',\ldots,M_\ell'$ of $G$ with the property that each homogeneous set $M$ of $G$ is a subset of one of $M_1',M_2',\ldots,M_\ell'$.
Let $M_2,M_3$ be two homogeneous sets in $G$ with $M_2\cap M_3\neq\emptyset$. For the sake of a contradiction, let us suppose that $M_2\cup M_3$ is not a homogeneous set in $G$. Hence, $M_2\setminus M_3, M_3\setminus M_2\neq\emptyset$, and we let $m_2\in M_2\setminus M_3$, $m_3\in M_3\setminus M_2$, and $m_4\in M_2\cap M_3$ be arbitrary vertices. Since $M_2$ and $M_3$ are modules, we have 
$$N_G(m_2)\setminus (M_2\cup M_3)=N_G(m_4)\setminus (M_2\cup M_3)=N_G(m_3)\setminus (M_2\cup M_3),$$
and so $N_G(m_2)\setminus (M_2\cup M_3)=\emptyset$ since $M_2\cup M_3$ is not a homogeneous set in $G$. Hence, $V(G)=M_2\cup M_3$ since $G$ is connected.
Clearly, $M_2\cap M_3$ is a module in $G$. Since $G$ has no clique-separators of modules, $M_2\cap M_3$ is not a clique-separator of modules, and so a vertex of $M_2\setminus M_3$ is adjacent to a vertex of $M_3\setminus M_2$. Hence, by the fact that $M_2$ and $M_3$ are modules, we have that each vertex of $M_2\cap M_3$ is adjacent to each vertex of $V(G)\setminus(M_2\cap M_3)$, and so $N_G^2(M_2\cap M_3)=\emptyset$, which contradicts the choice of $M_1$.
By this contradiction, $M_2\cup M_3$ is a homogeneous set for each two homogeneous sets in $G$ with $M_2\cap M_3\neq\emptyset$. 
Thus, let $M_1',M_2',\ldots,M_\ell'$ with $\ell\in\mathbb{N}_{0}$ be the inclusion-wise maximal homogeneous sets of $G$. As shown above, they are pairwise disjoint and each homogeneous set of $G$ is a subset of one of $M_1',M_2',\ldots,M_\ell'$.

Let $i\in[\ell]$ be arbitrary. We apply Lemma~\ref{preliminary_lemma_hom-sets}. Since, by the choice of $M_1$, $N_G^2(M_i')\neq\emptyset$ and since $G$ has no clique-separator of modules, it follows that $G[M_i']$ is $P_4$-free.
As shown by Seinsche \cite{Se} (see also the Strong Perfect Graph Theorem), $G[M_i']$ is perfect. By Lemma~\ref{preliminary_lemma_lovasz} and Observation~\ref{preliminary_obs_q_expansion}, it follows $\chi_q(G[M_i'])=\omega_q(G[M_i'])$. Since $q$ is $\ql{G}$-minimal, we obtain that $M_i'$ is a clique, and we fix a vertex $u_i'$ of $M_i'$ for each $i\in[\ell]$. 
Hence, let $k=1$ and $q_1\colon V(G)\to\mathbb{N}_0$ be a vertex-weight function with
$$u\mapsto\begin{cases}
q(M_i')&\text{if } u=u_i' \text{ for some } i\in[\ell],\\
0&\text{if } u\in M_i'\setminus\{u_i'\} \text{ for some } i\in[\ell],\\
q(u)&\text{if } u\notin \bigcup_{i=1}^\ell M_i'.\\
\end{cases}$$
Note that $q_1(v)=0$ if and only if $v\in M_i\setminus\{u_i'\}$ for some $i\in[\ell]$ as $G=G[q]$ and so $q(v)>0$. Thus, 
$q_1(G)=q(G)>0$.
In what follows, we show that $(G,q)$ is decomposable into the set $M_1$ and the vertex weight function $q_1$.  
First of all, we show (i). 
Clearly, $V(G[q_1])\subseteq M_1$.
As $G-((M_1'\cup M_2'\cup \ldots\cup M_\ell')\setminus\{u_1',u_2',\ldots,u_\ell'\})$ is prime and $q_1(v)>0$ for each of its vertices, it follows that $G[q_1]$ is prime as well.
For the sake of a contradiction, let us suppose that $X$ is a clique-separator of modules in $G[q_1]$. Since $G[q_1]$ is prime, every module of $X$ is of size $1$ and $X$ is simply a clique. Let 
$$X(x)=\begin{cases}
\{x\}&\text{if } x\notin \{u_1',u_2',\ldots,u_\ell'\},\\
M_i'&\text{if } x=u_i \text{ for some } i\in[\ell]\\
\end{cases}$$
for $x\in X$.
Since $M_1',M_2',\ldots,M_\ell'$ are pairwise disjoint modules which are cliques and for which $u_i'\in M_i'$ for each $i\in[\ell]$, $\bigcup_{x\in X} X(x)$ is a clique-separator of modules in $G$, which is a contradiction to the fact that, by Lemma~\ref{preliminary_lemma_cutsets} and the fact that $q$ is $\ql{G}$-minimal, such a set cannot exist.
Hence, (i) follows.
Furthermore, $G[M_1]$ is a strong expansion of $G[q_1]$, and so (ii) follows.
Clearly, (iii) is satisfied. 
It is further easily seen 
$$\omega_q(G)=\omega_q(G[M_1])=\omega_{q_1}(G)\quad\text{and}\quad \chi_q(G)=\chi_q(G[M_1])=\chi_{q_1}(G).$$
Thus, (iv) and (v) follow.
We conclude that $(G,q)$ is decomposable, which is a contradiction to our supposition on $(G,q)$. Therefore, our proof is complete.
\end{proof}

Lemma~\ref{preliminary_lemma_Q_P4} implies that, when studying $\chi$-binding functions in subclasses of $\Q{P_4}$-free graphs, the prime ones without clique-separators of modules are those of interest.

\begin{cor}\label{preliminary_corollary_main}
Let $G$ be a $\Q{P_4}$-free graph, $q\colon V(G)\to\mathbb{N}_0$ be a vertex-weight function, and $f\colon\mathbb{N}_{>0}\to\mathbb{N}_{>0}$ be a superadditive function. 
If $\chi_{q'}(G)\leq f(\omega_{q'}(G))$ for each $\ql{G}$-minimal vertex-weight function $q'\colon V(G)\to\mathbb{N}_0$ for which $G[q']$ is prime and has no clique-separator of modules, then
$$\chi_q(G)\leq f(\omega_q(G)).$$
\end{cor}
\begin{proof}
By Lemma~\ref{preliminary_lemma_Q_P4}, there is an integer $k\in\mathbb{N}_{>0}$ and there are $k$ $\ql{G}$-minimal vertex-weight functions $q_1,q_2,\ldots ,q_k\colon V(G)\to\mathbb{N}_0$ such that 
$$\chi_{q}(G)=\sum_{i=1}^k\chi_{q_i}(G)\quad\text{and}\quad\omega_q(G)\geq\sum_{i=1}^k\omega_{q_i}(G).$$
Furthermore, by Lemma~\ref{preliminary_lemma_Q_P4}, $G[q_i]$ is a prime graph without clique-separators of modules, and so
\linebreak
$\chi_{q_i}(G)\leq f(\omega_{q_i}(G))$
for each $i\in[k]$.
The superadditivity of $f$ implies
$$\chi_q(G)=\sum_{i=1}^k\chi_{q_i}(G)\leq \sum_{i=1}^k f(\omega_{q_i}(G))\leq f\left(\sum_{i=1}^k \omega_{q_i}(G)\right)\leq f(w_q(G)),$$
which completes our proof.
\end{proof}

Corollary~\ref{preliminary_corollary_main} obviously has a huge impact on studying $\chi$-binding functions. 
However, in view of its application, it is necessary that $\chi$-binding functions are superadditive. In what follows in the next lemma, we particularly show that $f^\star_{\{3K_1\}},f^\star_{\{C_5,3K_1\}}$, and $f^\star_{\{2K_2\}}$ are superadditive.

\begin{lemma}\label{preliminary_lemma_superadditiv}
Let $\mathcal{H}$ be a set of graphs such that no graph $H\in\mathcal{H}$ has a complete bipartite spanning subgraph.
If the class of $\mathcal{H}$-free graphs is $\chi$-bounded and $K_1$ is $\mathcal{H}$-free, then $f^\star_{\mathcal{H}}$ is superadditive.
\end{lemma}
\begin{proof}
We may assume that the class of $\mathcal{H}$-free graphs is $\chi$-bounded and $K_1$ is $\mathcal{H}$-free. Thus, $f^\star_{\mathcal{H}}(1)>0$. Let $w_1,w_2\geq 1$ be two arbitrary integers for which $f^\star_{\mathcal{H}}(\omega_1),f^\star_{\mathcal{H}}(\omega_2)>0$, $G'_1$ be an $\mathcal{H}$-free graph with $\omega(G'_1)=w_1$ and $\chi(G'_1)=f^\star_{\mathcal{H}}(w_1)$, and $G'_2$ be an $\mathcal{H}$-free graph with $\omega(G'_2)=w_2$ and $\chi(G'_2)=f^\star_{\mathcal{H}}(w_2)$ that is vertex disjoint from $G_1'$. 

Let $G$ be the graph obtained from $G'_1$ and $G'_2$ by adding all edges between the vertices of $G'_1$ and the vertices of $G'_2$. We prove first that $G$ is $\mathcal{H}$-free. For the sake of a contradiction, let us suppose that there is some $H\in\mathcal{H}$ for which $G$ contains a set $S$ of vertices inducing a copy of $H$. Since $G'_1$ and $G'_2$ are $H$-free, $s_1=|S\cap V(G'_1)|>0$ and $s_2=|S\cap V(G'_2)|>0$.
Therefore, the graph $G[S]$ has a spanning subgraph that is a copy of $K_{s_1,s_2}$. But now $G[S]\cong H$ gives a contradiction to our assumption that $H$ does not have a spanning subgraph which is a complete bipartite graph. Hence, $G$ is $\mathcal{H}$-free.

Clearly, $\omega(G)=w_1+w_2$ and $\chi(G)=\chi(G'_1)+\chi(G'_2)=f^\star_{\mathcal{H}}(w_1)+f^\star_{\mathcal{H}}(w_2)$, and so
$$f^\star_{\mathcal{H}}(w_1+w_2)\geq \chi(G)= f^\star_{\mathcal{H}}(w_1)+f^\star_{\mathcal{H}}(w_2)>0,$$
which completes our proof.
\end{proof}

Corollary~\ref{preliminary_corollary_main} is an important tool for determining $\chi$-binding functions in subclasses of $\Q{P_4}$-free graphs. However, as we wish to characterise critical graphs as well, we deduce the following from Lemma~\ref{preliminary_lemma_Q_P4}.

\begin{cor}\label{preliminary_cor_critrical}
If $G$ is a critical $\Q{P_4}$-free graph, then there is some integer $k\in\mathbb{N}_{>0}$ such that $V(G)$ can be partitioned into pairwise complete sets $M_1,M_2,\ldots,M_k$ such that $G[M_i]$ is a strong expansion of a prime graph without clique-separator of modules for each $i\in[k]$.
\end{cor}
\begin{proof}
Note that the vertex-weight function $q\colon V(G)\to[1]$ is $\ql{G}$-minimal since $G$ is critical. By Lemma~\ref{preliminary_lemma_Q_P4}, there exist an integer $k\in\mathbb{N}_{>0}$, $k$ pairwise disjoint non-empty sets $M_1,M_2,\ldots, M_k\subseteq V(G)$, and $k$ $\ql{G}$-minimal vertex-weight functions $q_1,q_2,\ldots, q_k\colon V(G)\to\mathbb{N}_0$ such that 
$V(G[q_i])\subseteq M_i$ and $G[M_i]$ is a strong expansion of $G[q_i]$ (and $G[q_i]$ is a prime graph without clique-separators of modules) for each $i\in[k]$,
$M_i$ is complete to $M_j$ in $G$ for each distinct $i,j\in [k]$, and
$$\chi(G)=\sum_{i=1}^k\chi(G[M_i]).$$
Since $G$ is critical, we conclude from the latter equality that $M_1,M_2,\ldots,M_k$ is indeed a partition of $V(G)$, which completes the proof.
\end{proof}

\section{$\Ha$-free graphs}

In this section, we prove Theorem~\ref{thm_main}~(ii) and Theorem~\ref{thm_critical_graphs}~(ii), and start showing the latter one.

\medskip

For the sake of a contradiction, let us suppose that $G$ is a critical $(P_5,\Ha)$-free graph that contains an induced copy of $2K_2$. As $G$ is critical, $G$ is connected. 
For two vertices $u,v\in V(G)$, we let $X_{u,v}=N_G(u)\cap N_G(v)$.

Let $u_1u_2$ be an arbitrary edge of $G$ such that $|E(G-N_G[\{u_1,u_2\}])|\geq 1$. 
If \linebreak $v\in N_G(\{u_1,u_2\})$, $w\in N_G(v)\cap N^2_G(\{u_1,u_2\})$, and $x\in N_G(w)\setminus N_G(\{u_1,u_2,v\})$, then, renaming vertices if necessary, we assume $u_1v\in E(G)$. Thus, $\{x,w,v,u_1,u_2\}$ induces a copy of $P_5$ if $u_2v\notin E(G)$ and a copy of $\Ha$ if $u_2v\in E(G)$, which is a contradiction to the fact that $G$ is $(P_5,\Ha)$-free. 
Hence, $N_G^i(\{u_1,u_2\})=\emptyset$ for $i\geq 3$, and each vertex subset of $N_G^2(\{u_1,u_2\})$ inducing a component of $G[N_G^2(\{u_1,u_2\})]$ is a module in $G$. 
Since $|E(G-N_G[\{u_1,u_2\}])|\geq 1$, there is some set $W$ of vertices which induces a component of $G-N_G[\{u_1,u_2\}]$ with at least one edge, say $w_1w_2$.
For each two adjacent vertices $w_3,w_4\in N_G^2(\{u_1,u_2\})$ and each $v\in N_G(\{u_1,u_2\})\cap N_G(\{w_3,w_4\})$, we have $v\in X_{u_1,u_2}\cap X_{w_3,w_4}$ since $\{u_1,u_2,v,w_3,w_4\}$ induces neither a copy of $\Ha$ nor a copy of $P_5$. Thus, for each set $W'$ of vertices inducing a component of $G-N_G[\{u_1,u_2\}]$ with at least one edge, say $w_3w_4$, we have that $W'$ is a module and $N_G(W')\subseteq X_{u_1,u_2}\cap X_{w_3,w_4}$. 
In particular, $N_G(W)\subseteq X_{u_1,u_2}\cap X_{w_1,w_2}$. Similarly, we obtain $N_G(U)\subseteq X_{u_1,u_2}\cap X_{w_1,w_2}$ for the set $U$ of vertices which induces the component of $G-N_G[\{w_1,w_2\}]$ which contains $u_1$ and $u_2$.
As $N_G(\{u_1,u_2\}),N_G(\{w_1,w_2\})\subseteq X_{u_1,u_2}\cap X_{w_1,w_2}$ (even equality holds by definition), we deduce that each of $U$ and $W$ induces a different component of 
$G-(X_{u_1,u_2}\cap X_{w_1,w_2})$.

Let $X_1,X_2,\ldots, X_k$ be the sets of vertices which induce the components of \linebreak $\bar{G}[X_{u_1,u_2}\cap X_{w_1,w_2}]$, and $i\in[k]$ be arbitrary.
We are going to show that $X_i$ is a module. For the sake of a contradiction, let us suppose that there is a vertex $y\in V(G)\setminus X_i$ with $X_i\cap N_G(y)\neq \emptyset$ and 
$X_i\setminus N_G(y)\neq \emptyset$. Clearly, $y\notin X_{u_1,u_2}\cap X_{w_1,w_2}.$ Since $\bar{G}[X_i]$ is connected, we may assume that $x_1\in X_i\cap N_G(y)$ and $x_2\in X_i\setminus N_G(y)$ are non-adjacent.
Let $Y$ be the set of vertices which induces the component of $G-(X_{u_1,u_2}\cap X_{w_1,w_2})$ that contains $y$.
If $|Y|=1$, then $u_1y \notin E(G)$ and $N_G(y)\subseteq N_G(u_1)$. Hence, $\chi_q(G-y)=\chi_q(G)$ since $y$ can be coloured by the colour of $u_1$ and we obtain a contradiction since $G$ is critical.
Thus, $|Y|\geq 2$ and there is a vertex $y'\in Y\cap N_G(y)$. 
As $u_1$ and $w_1$ are in different components of $G-(X_{u_1,u_2}\cap X_{w_1,w_2})$, we have $u_1,u_2\notin Y$ or $w_1,w_2\notin Y$. Renaming vertices if necessary, we may assume $u_1,u_2\notin Y$. Since $Y$ induces a component of $G-(X_{u_1,u_2}\cap X_{w_1,w_2})$, it is a module. Thus, $x_1y'\in E(G)$ but $x_2y'\notin E(G)$, and $\{x_2,u_1,x_1,y,y'\}$ induces a copy of $\Ha$, which is a contradiction to the fact that $G$ is $\Ha$-free.
Hence, $y$ does not exist, and $X_i$ is a module.
Let $Z_1=V(G)\setminus W$ and $Z_2=W\cup (X_{u_1,u_2}\cap X_{w_1,w_2})$. Clearly, $Z_1\cap Z_2=X_{u_1,u_2}\cap X_{w_1,w_2}$. 
Since $X_{u_1,u_2}\cap X_{w_1,w_2}$ is a clique-separator of the modules $X_1,X_2,\ldots, X_k$, we have 
$$\chi(G)=\max\{\chi(G[Z_1]),\chi(G[Z_2])\}$$
by Lemma~\ref{preliminary_lemma_cutsets}.
Since $u_1,u_2\in Z_1$ and $w_1,w_2\in Z_2$, 
we have that $G$ is not critical, which contradicts our assumption on $G$. Thus, $|E(G-N_G[\{v_1,v_2\}])|< 1$ for each edge $v_1v_2\in E(G)$, and so $G$ is $2K_2$-free, which is a contradiction to our supposition. Thus, Theorem~\ref{thm_critical_graphs}~(ii) follows.

\medskip

We continue and prove Theorem~\ref{thm_main}~(ii).
Since each $2K_2$-free graph is $(P_5,\Ha)$-free, it follows 
$$f^\star_{\{P_5,\Ha\}}(\omega)\geq f^\star_{\{2K_2\}}(\omega)$$
for each $\omega\in\mathbb{N}_{>0}$.
Furthermore, $f^\star_{\{2K_2\}}$ is superadditive by Lemma~\ref{preliminary_lemma_superadditiv}.
Let $G$ be a $(P_5,\Ha)$-free and $G'$ be a critical induced subgraph of $G$ such that $\chi(G')=\chi(G)$. By Theorem~\ref{thm_critical_graphs}~(ii), $G'$ is $2K_2$-free, and so
$$\chi(G)=\chi(G')\leq f^\star_{\{2K_2\}}(\omega(G'))\leq f^\star_{\{2K_2\}}(\omega(G)),$$
which completes our proof for Theorem~\ref{thm_main}~(ii).

\section{$\B$-free graphs}

This section is devoted to a proof of Theorem~\ref{thm_main} (i) and (iii), and Theorem~\ref{thm_critical_graphs} (i) and (iii).

\medskip

In order to prove Theorem~\ref{thm_critical_graphs} (i) and (iii), let $G$ be a critical $\B$-free graph which is $(C_5,C_7,\ldots)$-free or $P_5$-free.
By Corollary~\ref{preliminary_cor_critrical}, the vertex set of $G$ can be partitioned into $k\geq 1$ sets $M_1,M_2,\ldots,M_k$ such that $G[M_i]$ is a strong expansion of a prime graph $G_i$ for each $i\in[k]$ and $M_i$ is complete to $M_j$ in $G$ for each distinct $i,j\in[k]$.
Following the next two theorems, $G_i$ is $3K_1$-free or perfect for each $i\in[k]$. 
\begin{theorem}[Ho\'{a}ng \cite{H}]\label{banner_C_5}
If $G$ is a prime $(C_5,C_7,\ldots,\B)$-free graph of independence number at least $3$, then $G$ is perfect.
\end{theorem}
\begin{theorem}[Karthick, Maffray, and Pastor \cite{KMP}]\label{banner_P_5}
If $G$ is a prime $(P_5,\B)$-free graph of independence number at least $3$, then $G$ is perfect.
\end{theorem}
As an immediate consequence $G[M_i]$ is $3K_1$-free or, by Lemma~\ref{preliminary_lemma_lovasz}, $G[M_i]$ is perfect. 
In the latter case, $G[M_i]$ is complete since $G$ is critical. Thus, in both cases, $G[M_i]$ is $3K_1$-free.
Since $M_i$ is complete to $M_j$ in $G$ for each distinct $i,j\in[k]$, $G$ is $3K_1$-free as well.
Hence, our proof for Theorem~\ref{thm_critical_graphs}~(i) and (iii) is complete.

\medskip

We continue and prove Theorem~\ref{thm_main} (i) and (iii). We note that each graph of $\{C_7,C_9,\ldots,P_5,\B\}$ contains at least one induced copy of $3K_1$.
Consequently, for each $\omega\geq 1$, we have
$$f^\star_{\{C_5,C_7,\ldots,\B\}}(\omega)\geq f^\star_{\{C_5,3K_1\}}(\omega)\quad\text{and}\quad f^\star_{\{P_5,\B\}}(\omega)\geq \f(\omega).$$
Let $G$ be a $\B$-free graph that is $(C_5,C_7,\ldots)$-free or $P_5$-free and $G'$ be a critical induced subgraph of $G$ such that $\chi(G')=\chi(G)$. 
We find that $G'$ is $3K_1$-free by Theorem~\ref{thm_critical_graphs} (i) and (iii).
Furthermore, by Lemma~\ref{preliminary_lemma_superadditiv}, $f^\star_{\{C_5,3K_1\}}$ and $f^\star_{\{3K_1\}}$ are superadditive.
It follows from $\omega(G')\leq\omega(G)$ that
$$\chi(G)=\chi(G')\leq f^\star_{\{C_5,3K_1\}}(\omega(G'))\leq f^\star_{\{C_5,3K_1\}}(\omega(G))$$
if $G$ is $(C_5,C_7,\ldots)$-free
and
$$\chi(G)=\chi(G')\leq f^\star_{\{3K_1\}}(\omega(G'))\leq f^\star_{\{3K_1\}}(\omega(G))$$
if $G$ is $P_5$-free.
Hence, our proof for Theorem~\ref{thm_main} (i) and (iii) is complete.

\section{$C_4$-free graphs}

This section is devoted to a proof of Theorem~\ref{thm_main} (iv) and Theorem~\ref{thm_critical_graphs} (iv).

\medskip

In order to prove Theorem~\ref{thm_critical_graphs} (iv), we may assume that $G$ is a critical $(P_5,C_4)$-free graph.
By Corollary~\ref{preliminary_cor_critrical}, there exist an integer $k\in\mathbb{N}_{>0}$ such that $V(G)$ can be partitioned into $k$ sets $M_1,M_2,\ldots, M_k\subseteq V(G)$ such that 
$M_i$ is complete to $M_j$ in $G$ for distinct $i,j\in[k]$, and $G[M_i]$ is a strong expansion of a prime graph $G_i$ without clique-separator of modules for each $i\in[k]$.
Let us assume $\alpha(G[M_1])\geq \alpha(G[M_i])$ for each $i\in[k]$. 
If $\alpha(G[M_1])=1$, then $G=G[M_1\cup M_2\cup\ldots\cup M_k]$ is complete, and the desired result follows.
In what follows we assume $\alpha(G[M_1])\geq 2$.
Since $G$ is $C_4$-free, we have that $V(G)\setminus M_1$ is a clique in $G$, and so $G-M_1$ is complete and a strong expansion of $G[\{u\}]$ for some $u\in V(G)\setminus M_1$.
We note that $G[M_1]$ is critical. From the fact $\alpha(G[M_1])\geq 2$, it follows $\chi(G[M_1])>\omega(G[M_1])$. Thus, $G[M_1]$ is not perfect.
Additionally, $G_1$ is isomorphic to an induced subgraph of $G[M_1]$, and so it is a prime $(P_5,C_4,\B)$-free graph. In particular, $G_1$ is $3K_1$-free by Theorem~\ref{banner_P_5}. Hence, $\bar{G}_1$ is non-bipartite by the Strong Perfect Graph Theorem and $(2K_2,C_3)$-free. 
Randerath's~\cite{R} characterisation of non-bipartite $(P_5,C_3)$-free graphs implies that the prime ones are copies of $C_5$, and so $G_1\cong C_5$. 
Thus, there is a vertex-weight function $q'\colon V(G')\to\mathbb{N}_{>0}$ such that $G$ is a $q'$-expansion of $G'\in\{C_5,W_5\}$.
Therefore, Theorem~\ref{thm_critical_graphs}~(iv) follows.

\medskip

We continue and prove Theorem~\ref{thm_main}~(iv). Note that every $q$-expansion of $C_5$ for a vertex-weight function $q\colon V(C_5)\to\mathbb{N}_0$ is $(P_5,C_4)$-free. By Observation~\ref{preliminary_obs_q_expansion} and Corollary~\ref{preliminary_cor_C5}, we have
$$f_{\{P_5,C_4\}}^\star(\omega)\geq 
\left\lceil\frac{5\omega-1}{4}\right\rceil
$$
for each $\omega\in\mathbb{N}_{>0}$.
Hence, it suffices to prove $\chi(G)\leq \lceil(5\omega(G)-1)/4\rceil$ for each $(P_5,C_4)$-free graph.
Let $G'$ be a critical induced subgraph of $G$ such that $\chi(G')=\chi(G)$.
If $G'$ is complete, then 
$$\chi(G)=\chi(G')=\omega(G')\leq \omega(G)\leq \left\lceil\frac{5\omega(G)-1}{4}\right\rceil.$$
Hence, we may assume that $G'$ is not complete.
By Theorem~\ref{thm_critical_graphs}~(iv), there is a vertex-weight function $q'\colon V(W_5)\to\mathbb{N}_0$ such that $G'$ is a $q'$-expansion of $W_5$.
In particular, $q'(W_5-w)>0$ for the vertex $w$ of maximum degree in~$W_5$ as otherwise $q'(w)>0$ and $G'$ is a complete graph. Furthermore, $w$ is complete to $V(W_5)\setminus\{w\}$ in $G$.
Thus,
\begin{align*}
\chi(G)&=\chi(G')=\chi_{q'}(W_5)=\chi_{q'}(W_5[\{w\}])+\chi_{q'}(W_5-w)\leq q'(w)+\left\lceil\frac{5\omega_{q'}(W_5-w)-1}{4}\right\rceil \\
& \leq \left\lceil\frac{5(\omega_{q'}(W_5-w)+q'(w))-1}{4}\right\rceil=\left\lceil\frac{5\omega_{q'}(W_5)-1}{4}\right\rceil=\left\lceil\frac{5\omega(G')-1}{4}\right\rceil \leq \left\lceil\frac{5\omega(G)-1}{4}\right\rceil
\end{align*}
by Observation~\ref{preliminary_obs_q_expansion} and Corollary~\ref{preliminary_cor_C5},
which completes our proof for Theorem~\ref{thm_main} (iv).

\subsubsection*{Acknowledgement}
We thank both reviewers for their helpful comments which greatly improved our paper.

\end{document}